\documentclass[10pt,a4paper]{amsart}

\usepackage{a4wide}
\usepackage{amsthm}
\usepackage{amsfonts}
\usepackage{amsmath}
\usepackage[english]{babel}
\usepackage{mathtools}

\usepackage{faktor}
\usepackage{etex}
\usepackage{hyperref}
\usepackage{multimedia}
\usepackage[all]{xy}
\usepackage{amssymb}
\usepackage{graphicx,textpos}
\usepackage[dvipsnames]{xcolor}
\usepackage{helvet}
\usepackage{stmaryrd}
\usepackage{enumerate}
\usepackage{todonotes}
\usepackage{pdfsync}
\usepackage{dsfont}
\usepackage[shortcuts]{extdash}
\usepackage[all]{xy}
  \newdir{ >}{{}*!/-9pt/@{>}}

\def\backslash{\delimiter"526E30F\mathopen{}}

\newcommand{\N}{\mathbb{N}}
\newcommand{\R}{\mathbb{R}}
\newcommand{\Q}{\mathbb{Q}}
\newcommand{\Z}{\mathbb{Z}}

\newcommand{\I}{\mathds{1}}

\newcommand{\T}{\mathbb{T}}
\newcommand{\rst}[1]{\ensuremath{{\mathbin\mid}\raise-.5ex\hbox{$#1$}}}
\newcommand{\lie}{\mathfrak{g}}

\newcommand{\lien}{\mathfrak{n}}

\newcommand{\suchthat}{\;\ifnum\currentgrouptype=16 \middle\fi|\;}

\newcommand{\compcent}[1]{\vcenter{\hbox{$#1\circ$}}}
  
\newcommand{\comp}{\mathbin{\mathchoice
    {\compcent\scriptstyle}{\compcent\scriptstyle}
    {\compcent\scriptscriptstyle}{\compcent\scriptscriptstyle}}}

\DeclareMathOperator{\GL}{GL}

\DeclareMathOperator{\Aut}{Aut}
\DeclareMathOperator{\Aff}{Aff}
\DeclareMathOperator{\aff}{aff}

\DeclareMathOperator{\Endo}{Endo}
\DeclareMathOperator{\Per}{Per}
\DeclareMathOperator{\ePer}{ePer}
\DeclareMathOperator{\Fix}{Fix}

\DeclareMathOperator{\ordz}{ord}
\newcommand{\ord}{\ordz_H}
\newcommand{\ordn}{\ordz_N}
\newcommand{\ordt}{\ordz_{\Z^2}}

\let\emptyset\varnothing

\author{Jonas Der\'e}
\address{KU Leuven Kulak, E. Sabbelaan 53, BE-8500 Kortrijk, Belgium}
\email{jonas.dere@kuleuven-kulak.be}
\thanks{The author is supported by a Ph.D.~fellowship of the Research Foundation -- Flanders (FWO). Research supported by the research Fund of the KU Leuven}
\title[Periodic points of affine infra-nilmanifold endomorphisms]{\bf Periodic and eventually periodic points of affine infra-nilmanifold endomorphisms}
\date{\today}

\makeatletter
\newtheorem*{rep@theorem}{\rep@title}
\newcommand{\newreptheorem}[2]{%
\newenvironment{rep#1}[1]{%
 \def\rep@title{#2 \ref{##1}}%
 \begin{rep@theorem}}%
 {\end{rep@theorem}}}
\makeatother

\newtheorem{Thm}{Theorem}[section]
\newreptheorem{theorem}{Theorem}
\newtheorem{Lem}[Thm]{Lemma}
\newtheorem{Prop}[Thm]{Proposition}
\newtheorem{Cor}[Thm]{Corollary}

\theoremstyle{definition}
\newtheorem{Def}[Thm]{Definition}
\newtheorem{Ex}[Thm]{Example}

\theoremstyle{remark}
\newtheorem{Rmk}[Thm]{Remark}
\newtheorem{QN}{Question}

\numberwithin{equation}{section}

\begin{document}
\begin{abstract}
In this paper, we study the periodic and eventually periodic points of affine infra-nilmanifold endomorphisms. On the one hand, we give a sufficient condition for a point of the infra-nilmanifold to be (eventually) periodic. In this way we show that if an affine infra-nilmanifold endomorphism has a periodic point, then its set of periodic points forms a dense subset of the manifold. On the other hand, we deduce a necessary condition for eventually periodic points from which a full description of the set of eventually periodic points follows for an arbitrary affine infra-nilmanifold endomorphism . 
\end{abstract}
\maketitle

Infra-nilmanifolds play an important role when studying certain dynamical systems. For example, M.~Gromov completed the proof that every manifold admitting an expanding map is homeomorphic to an infra-nilmanifold by showing that every group of polynomial growth is virtually nilpotent in \cite{grom81-1}. It is conjectured that the same result is true for manifolds admitting an Anosov diffeomorphisms and there are some partial results, for instance in the case of codimension one in \cite{newh70-1}. Expanding maps and Anosov diffeomorphisms were among the first examples of dynamical systems which are structural stable, meaning that small perturbations of these maps are topologically conjugate to the original map. For more details about these dynamical systems, we refer to seminal paper \cite{smal67-1}.

Expanding maps are also examples of chaotic dynamical systems as was proven in \cite{shub69-1}. One of the defining properties of such dynamical systems, see \cite{deva89-1}, is that the periodic points form a dense subset of the manifold. For general Anosov diffeomorphisms, it is still an open question whether the periodic points are dense, although there are some partial results, e.g.\ for $C^2$ Anosov diffeomorphisms in \cite{anos69-1} and for Anosov diffeomorphisms on infra-nilmanifolds in \cite{mann74-1}. 

In this paper, we study the periodic points and eventually periodic points of a more general class of self-maps on infra-nilmanifolds than expanding maps or Anosov diffeomorphisms, namely the affine infra-nilmanifold endomorphisms. These maps are induced by an affine transformation of the covering Lie group and are therefore the easiest examples of self-maps on infra-nilmanifolds. It is a result of K.B.~Lee in \cite{lee95-2} that every self-map of an infra-nilmanifold is homotopic to an affine infra-nilmanifold endomorphism and thus they form a rich class of self-maps on infra-nilmanifolds. In fact, the exact statement of M.~Gromov is that every expanding map is topological conjugate to an affine infra-nilmanifold endomorphism.

In \cite{hkl10-1}, the authors show that the set of eventually periodic points of an infra-nilmanifold endomorphism (which is induced by an automorphism on the covering Lie group) is dense in the infra-nilmanifold. The main idea of this paper is to construct subsets of eventually periodic points by using the relation between Lie algebras and Lie groups given by the Baker-Campbell-Hausdorff formula. In this paper, we improve the results of \cite{hkl10-1} in three directions. 

First of all, we study the more general class of affine infra-nilmanifold endomorphisms in this paper. This means that the maps we consider are not necessarily induced by an automorphism of the covering Lie group, but by an affine transformation. Secondly we show that not only the eventually periodic points, but also the smaller set of periodic points forms a dense subset for every affine infra-nilmanifold endomorphism which has at least one periodic point:
\begin{reptheorem}{maindense}
Let $\bar{\alpha}$ be an affine infra-nilmanifold endomorphism of the infra-nilmanifold $\Gamma \backslash G$. Then either $\Per(\bar{\alpha}) = \emptyset$ or $\Per(\bar{\alpha})$ is a dense subset of $\Gamma \backslash G$.
\end{reptheorem}
\noindent Finally, we give an explicit description of the set of eventually periodic points for a general affine infra-nilmanifold endomorphism in Theorem \ref{fulldes}.

This paper is structured as follows. First, we recall the basis properties of infra-nilmanifolds and affine infra-nilmanifold endomorphisms. Next we show how the general case of affine infra-nilmanifold endomorphisms reduces to the more restrictive case of nilmanifold endomorphisms. In the next two sections we then give sufficient and necessary conditions for points of the infra-nilmanifold to be (eventually) periodic in this special case of nilmanifold endomorphisms, leading to the main results of this paper. At the end, we show how these results can be generalized to a bigger class of self-maps.

\section{Affine infra-nilmanifold endomorphisms}

\label{intro}

We start by recalling some basic properties about infra-nilmanifolds and affine infra-nilmanifold endomorphisms. Standard references for nilpotent groups, almost-crystallographic groups and affine infra-nilmanifold endomorphisms are \cite{sega83-1}, \cite{deki96-1} and \cite{deki11-1} respectively.

Let $G$ be a connected and simply connected nilpotent Lie group and $\Aut(G)$ the group of continuous automorphisms of the Lie group $G$. Consider the affine group $\Aff(G)$ of $G$, which is the semidirect product $G \rtimes \Aut(G)$. The natural action of the group $\Aff(G)$ on $G$ is given by 
\begin{eqnarray*}
\forall \alpha = (g, \delta) \in \Aff(G), \hspace{1mm} \forall h \in G:  {}^\alpha h = g \delta(h).
\end{eqnarray*}
Fix a finite group $F \le \Aut(G)$ of automorphisms of $G$. An almost-crystallographic group $\Gamma$ is a discrete subgroup of $G \rtimes F \le \Aff(G)$ such that the quotient space $\Gamma \backslash G$ is compact. If $\Gamma$ is moreover torsion-free then we call $\Gamma$ an almost-Bieberbach group. In this latter case, the quotient space $\Gamma \backslash G$ is a closed manifold which we call an \textbf{infra-nilmanifold} modeled on the Lie group $G$. 

The group $G$ is always identified with the subgroup of pure translations in $\Aff(G)$. If $\Gamma \le G$ or equivalently if $\Gamma$ is a lattice of the Lie group $G$, the manifold $\Gamma \backslash G$ is called a nilmanifold. Note that in this case $\Gamma$ is a torsion-free, finitely generated nilpotent group and a group satisfying these properties is called an $\mathcal{F}$-group. Every $\mathcal{F}$-group occurs as a lattice of a simply connected and connected nilpotent Lie group, so these three poperties determine the fundamental groups of nilmanifolds completely. If $\Gamma \le G$ is a lattice and $\varphi: \Gamma \to \Gamma$ a group morphism, then there exists a unique extension of $\varphi$ to the Lie group $G$. 

Let $p: \Gamma \to F$ be the natural projection on the second component, then we call $p(\Gamma)$ the holonomy group of the infra-nilmanifold $\Gamma \backslash G$. We will always assume that $p$ is surjective or thus that $F$ is equal to the holonomy group of $\Gamma \backslash G$. By construction every almost-Bieberbach group $\Gamma$ fits in the short exact sequence $$\xymatrix{ 1 \ar[r] & N \ar[r] & \Gamma \ar[r] & F \ar[r] &1}$$ where $N = G \cap \Gamma$ is the subgroup of pure translation of $\Gamma$. The normal subgroup $N$ is the maximal normal nilpotent subgroup and thus equal to the Fitting subgroup of $\Gamma$. In particular, every infra-nilmanifold $\Gamma \backslash G$ is finitely covered by the nilmanifold $N \backslash G$. The standard examples of infra-nilmanifolds start from an abelian Lie group $G = \R^n$ and the manifolds we obtain in this way are exactly the flat manifolds where nilmanifolds correspond to the flat tori $\Z^n \backslash \R^n$. 

Let $\alpha \in \Aff(G)$ be an affine transformation such that $\alpha \Gamma \alpha^{-1} \le \Gamma$. Then $\alpha$ induces a map $\bar{\alpha}: \Gamma \backslash G \to \Gamma \backslash G$ given by $$ \bar{\alpha}(\Gamma g) = \Gamma {}^\alpha g.$$ This map is indeed well-defined, since for every $\gamma \in \Gamma$ we have that $${}^\alpha \left( {}^\gamma g \right) = {}^{\alpha \gamma} g = {}^{ \tilde{\gamma} \alpha} g = {}^{\tilde{\gamma}} \left( {}^\alpha g \right)$$ with $\tilde{\gamma} = \alpha \gamma \alpha^{-1} \in \Gamma$. The maps induced in this way are the self-maps we study in this paper.

\begin{Def}
Let $\Gamma \backslash G$ be an infra-nilmanifold and $\alpha \in \Aff(G)$ an affine transformation with $\alpha \Gamma \alpha^{-1} \le \Gamma$. The induced map $\bar{\alpha}: \Gamma \backslash G \to \Gamma \backslash G$ is called an \textbf{affine infra-nilmanifold endomorphism}. If $\alpha \in \Aut(G)$, then we call $\bar{\alpha}$ an \textbf{infra-nilmanifold endomorphism}. If $\Gamma \backslash G$ is moreover a nilmanifold, we call the map a \textbf{nilmanifold endomorphism}.
\end{Def}

An interesting question is how the set of periodic points or eventually periodic points of such an affine infra-nilmanifold endomorphism look like. 
\begin{Def}
Let $f: X \to X$ be an endomorphism of a space $X$. 
\begin{itemize}
\item A point $x \in X$ is called a \textbf{fixed} point of $f$ if it satisfies $f(x)= x$.
\item A point $x \in X$ is a \textbf{periodic} point for $f$ if there exists $k >0$ such that $x$ is a fixed point of $f^k$. The subset of all periodic points of $f$ is denoted as $\Per(f) \subseteq X$.
\item A point $x \in X$ is called \textbf{eventually periodic} if $f^k(x)$ is periodic for some $k > 0$ and the set of eventually periodic points is denoted as $\ePer(f)$. 
\end{itemize}  
\end{Def} \noindent Note that $\Per(f) = \Per(f^k)$ and $\ePer(f) = \ePer(f^k)$ for every $k > 0$. In this paper, we give a partial answer to the question how the (eventually) periodic points of an affine infra-nilmanifold endomorphism looks like, which implies already that the periodic points form a dense subset. For these results we still need the notion of rational Mal'cev completion or radicable hull of an $\mathcal{F}$-group.

Assume that $N$ is an $\mathcal{F}$-group and take $G$ a simply connected and connected nilpotent Lie group such that $N \le G$ is a lattice. Consider the Lie algebra $\lie$ corresponding to $G$ with exponential map $\exp: \lie \to G$. In this case, $\exp$ is a diffeomorphism and denote by $\log$ its inverse. The rational subspace spanned by $\log(N) \subset \lie$ is a rational Lie algebra which we denote as $\lien^\Q$. Under the exponential mapping, this rational Lie algebra corresponds to a nilpotent group $N^\Q$ which is radicable, i.e. for every $n \in N^\Q$ and $k > 0$, there exists an $m \in N^\Q$ such that $m^k = n$. The group $N^\Q$ is called the radicable hull of the $\mathcal{F}$-group $N$. In torsion-free nilpotent groups, for every $k > 0$ we have that $x^k = y^k$ if and only if $x = y$. So if $N^\Q$ is a radicable torsion-free nilpotent group, then for every $k > 0$ and $n \in N^\Q$ there exists a unique $m \in N^\Q$ such that $m^k = n$ and we write this as $m = n^{\frac{1}{k}}$. 

Let us end this introduction by recalling some definitions about covering maps and the relation to the fundamental group from \cite{munk75-1}. Let $p: E \to B$ a continuous surjection between connected spaces $E$ and $B$, then we say that an open subset $U \subseteq B$ is evenly covered by $p$ if $p^{-1}(U)$ is a disjoint union of open subsets, each of which is mapped homeomorphically onto the subset $U$ by $p$. If every point of $B$ has an open neighborhood which is evenly covered by $p$, then we call $p$ a covering map. A covering map $p: E \to B$ is called finite if $p^{-1}(b)$ is finite for one (and hence for every) $b \in B$.

We call $f: E \to E$ a lift of the map $g: B \to B$ if the following diagram commutes:
$$ \xymatrix{
E \ar[r]^f \ar[d]^p & E \ar[d]^p \\
B \ar[r]^g 			& B.} $$ 
A covering transformation of $p$ is a homeomorphism $h: E \to E$ such that $p \comp h = p$ or thus a lift of the identity map $\I_B$. The set of all covering transformations forms a group under composition. A lift $f: E \to E$ of the map $g: B \to B$ is not unique, but every other lift of $g$ is of the form $h \comp f$ for $h$ a covering transformation.

We recall how the group $\Gamma$ of covering transformations of the universal cover $p: \bar{B} \to B$ is isomorphic to the fundamental group. Fix a point $b \in B$ and $b_0 \in \bar{B}$ such that $p(b_0) = b$. Take $\gamma \in \Gamma$ and consider a path $f: I \to \bar{B}$ from $b_0$ to ${}^\gamma b_0$. This path projects to a loop $p \comp f$ starting in $b$ and thus to an element of the fundamental group $[p\comp f]_h \in \pi_1(B,b)$. Since $\bar{B}$ is simply connected, every two paths from $b_0$ to ${}^\gamma b_0$ are homotopic. This implies that the element $[p\comp f]_h$ does not depend on the choice of the path $f$. So we have a well-defined map $$\varphi_{b,b_0}: \Gamma \to \pi_1(B,b)$$ and it is a general fact that this map is a group morphism which forms an isomorphism between the group $\Gamma$ and $\pi_1(B,b)$. In this paper, we always identify the fundamental group with the group $\Gamma$ of covering transformations. 

We call the covering map $p: E \to B$ regular if for every $e \in E$ and $e^\prime \in p^{-1}(p(e))$, there exists a covering transformation $h$ such that $h(e) = e^\prime$. This property is equivalent to $p_\ast(\pi_1(E,e))$ being a normal subgroup of $\pi_1(B,p(e))$ for some $e \in E$ and in this case, the group of covering transformations is isomorphic to $\faktor{\pi_1(B,p(e))}{p_\ast\left(\pi_1(E,e)\right)}$. 

Let $g: B \to B$ be a continuous map and consider a lift $f: \bar{B} \to \bar{B}$. Note that for every $\gamma \in \Gamma$, the map $f \comp \gamma$ is also a lift of $g$ and therefore there exists a unique $f^{\#}(\gamma) \in \Gamma$ such that $f \comp \gamma = f^\#(\gamma) \comp f$. The map $f^\#: \Gamma \to \Gamma$ is a group morphism which makes the following diagram commute:
 $$ \xymatrix{
\Gamma \ar[r]^{f^\#} \ar[d]^{\varphi_{b,b_0}} & \Gamma \ar[d]^{\varphi_{g(b),f(b_0)}} \\
\pi_1(B,b) \ar[r]^{g_\ast} 			& \pi_1(B,g(b))} $$
for every $b \in B$ and $b_0 \in \bar{B}$ such that $p(b_0) = b$. Therefore, under the identification between $\Gamma$ and the fundamental group, the group morphism $f^{\#}$ corresponds to the induced map $g_\ast$ on the fundamental group. The advantage of working in this way is that the map $f^{\#}$ does not depend on the choice of a base point and we will omit the notation of the base point while using this identification. The map $f^{\#}$ does depend on the choice of lift of $g$, but in the cases we consider in this paper, there will be a canonical choice for the lift $f$. If $f$ is a diffeomorphism, then $f^\#(\gamma) = f \comp \gamma \comp f^{-1}$.

\section{Reduction to nilmanifold endomorphisms}
\label{reduction2}
In this section, we show how studying (eventually) periodic points of affine infra-nilmanifold endomorphisms can be reduced to the more restrictive case of nilmanifold endomorphisms. These type of maps are induced by automorphisms on nilmanifolds and are therefore easier to handle. Their (eventually) peroidic points are then studied in more detail in the following sections.

\smallskip

We start with the following lemma which, although it is very elementary, forms an important ingredient for the main results of this paper.
\begin{Lem}
\label{elem}
Let $f: S \to S$ be a map on a finite set $S$. Then we have $\Per(f) \neq \emptyset$ and $\ePer(f) = S$. If $f$ is moreover injective then $\Per(f) = S$.
\end{Lem}
\begin{proof}
Take any $s \in S$ and consider its orbit $$\left\{ f^k(s) \suchthat k > 0\right\}\subseteq S.$$ Since $S$ is finite, there exists $k_1 < k_2$ such that $f^{k_1}(s) = f^{k_2}(s)$. This implies that $f^{k_1}(s)$ is a periodic point of $f$ and thus $\Per(f) \neq \emptyset$. The element $s \in \ePer(f)$ as well and since $s$ was chosen arbitrary, this implies $\ePer(f) = S$. If $f$ is injective, then $f$ forms a permutation of the finite set $S$, hence $f$ has finite order and $\Per(f) = S$.
\end{proof}
To main idea of the paper is to construct finite subsets of the infra-nilmanifold which are preserved by the affine infra-nilmanifold endomorphism. In this way, Lemma \ref{elem} helps us constructing (eventually) periodic points. The following proposition shows that Lemma \ref{elem} allows us to study the behavior of (eventually) periodic points under finite covering maps as well.
\begin{Prop}
\label{cover}
Let $p: E \to B$ a finite covering map and consider maps $f: E \to E$ and $g: B \to B$ such that $f$ is a lift of $g$. Then the following statements hold:
\begin{enumerate}[$(1)$]
\item $p^{-1}\left(\ePer\left(g\right)\right) = \ePer(f)$;
\item $p(\Per(f)) = \Per(g)$;
\item If $E$ is a regular covering map and $$g_\ast: \faktor{\pi_1(B)}{\pi_1(E)} \to \faktor{\pi_1(B)}{\pi_1(E)}$$ is injective, then $p^{-1}(\Per(g)) = \Per(f)$.
\end{enumerate}
\end{Prop}
\begin{proof}
For the first statement, the inclusion $\ePer(f) \subseteq p^{-1}(\ePer(g))$ follows immediately from the fact that $f$ is a lift of the map $g$. For the other inclusion, take any $b \in \ePer(g)$, then there exists $n,k >0$ such that $g^{n+k}(b) = g^{n}(b)$. For every $e \in p^{-1}(g^n(b))$, we have that $$p(f^k(e)) = g^k(p(e)) = g^k(g^n(b)) = g^n(b) = p(e),$$ and thus $f^k$ preserves the finite set $p^{-1}(g^n(b))$. By Lemma \ref{elem} we conclude that every point of $p^{-1}(g^n(b))$ is eventually periodic for the map $f^k$ and thus also for the map $f$. Since every $e \in p^{-1}(b)$ satisfies $f^n(e) \in p^{-1}(g^n(b))$, this implies that $p^{-1}(b) \subseteq \ePer(f)$.

The inclusion $p(\Per(f)) \subseteq \Per(g)$ of the second statement follows again from the definition of lift of a map. For the other inclusion, take $b \in \Per(g)$ and $k >0$ such that $g^k(b) = b$. Similarly as above, $f^k$ preserves the finite set $p^{-1}(b)$ and by Lemma \ref{elem}, we conclude that there exists $e \in p^{-1}(b)$ such that $e \in \Per(f^k) = \Per(f)$ or thus that $\Per(g) \subseteq p(\Per(f))$. 

For the last statement, we only have to show that $f^k$ is injective on $p^{-1}(b)$ in the previous argument. Denote by $H$ the finite group of deck transformations of $p$ which is isomorphic to the group  $\faktor{\pi_1(B)}{\pi_1(E)}$ as explained in Section \ref{intro}. Under this isomorphism, the morphism $g_\ast$ corresponds to the morphism $f^\#$ and thus $f^k$ induces an injective group morphism $$\varphi := \left(f^k\right)^\#: H \to H.$$ Fix an element $e_0 \in p^{-1}(b)$, then every other element $e \in p^{-1}(b)$ is uniquely represented as $e = h( e_0)$ for some $h \in H$. From the definition of $f^\#$, see Section \ref{intro}, it follows that $$f^k(e) = f^k(h(e_0)) = \varphi(h) \left(f^k(e_0)\right).$$ This implies that $f^k$ is injective on the set $p^{-1}(b)$ and thus every $e \in p^{-1}(b)$ is a periodic point of $f^k$ and therefore also of $f$. \end{proof}

There is an important difference between statements $(1)$ and $(2)$ of Proposition \ref{cover}, which we illustrate with the following example.

\begin{Ex}
\label{expand}
Consider the torus $B = \Z^n \backslash \R^n$ and take integers $k_1, \ldots, k_n$ with $\vert \prod_{i=1}^n k_i \vert > 1$. Let $\delta$ be the linear map $\R^n \to \R^n$ induced by the matrix $$\begin{pmatrix} k_1 & 0 & \hdots & 0\\0 & k_2 & \hdots & 0\\ \vdots & \vdots & \ddots & \vdots \\ 0 & 0 & \hdots & k_n \end{pmatrix} \in \GL(n,\Q).$$ Consider the lattice $L = \delta(\Z^n) \le \Z^n$ of $\R^n$ and denote the torus $L \backslash \R^n$ as $E$. Since $L$ is a finite index normal subgroup of $\Z^n$, $E$ is a regular finite cover of $B$ and take $p: E \to B$ the covering map. It holds that $\delta(\Z^n) \le \Z^n$ and $\delta(L) \le \delta(\Z^n) = L$ and thus $\delta$ induces a nilmanifold endomorphism on both nilmanifolds $B$ and $E$. We denote these induced nilmanifold endomorphisms as $f: E \to E$ and $g: B \to B$. Note that $f$ is indeed a lift of the map $g$. 

Denote by $q$ the universal covering map $\R^n \to E$, then $p \comp q$ is the universal covering map for $B$. Every point of $X = q(\Z^n) \subseteq E$ is an eventually periodic point of $f$, since $$f^m(q(z)) = q (\delta^m(z)) \in q(L) = \{e L\}$$ for every $m > 0$ and $z \in \Z^n$. But only the point $q(0) \in X$ is a periodic point of the map $f$. The set $X$ has $\vert \det(\delta) \vert = \vert \prod_{i=1}^n k_i \vert > 1$ elements and so not every point in $X \subseteq p^{-1}(\Per(g)) $ is periodic, although $p(\Per(f)) = \Per(g)$.
\end{Ex}

By applying this result to affine infra-nilmanifold endomorphisms we get the following consequence of Proposition \ref{cover}.

\begin{Thm}
\label{reduc}
Let $\bar{\alpha}$ be an affine infra-nilmanifold endomorphism on the infra-nilmanifold $\Gamma \backslash G$ induced by $\alpha \in \Aff(G)$. Let $N \triangleleft \Gamma$ be the Fitting subgroup of $\Gamma$ and consider the natural covering map $p: N \backslash G \to \Gamma \backslash G$. Then $\alpha$ induces an affine nilmanifold endomorphism $\tilde{\alpha}$ on $N \backslash G$ which is the lift of $\bar{\alpha}$ and the (eventually) periodic points satisfy 
\begin{align*}
p^{-1}\left(\ePer(\bar{\alpha})\right)& =\ePer(\tilde{\alpha})\\
p^{-1}\left(\Per(\bar{\alpha})\right)& =\Per(\tilde{\alpha}).
\end{align*}
\end{Thm}

\begin{proof}
We identify the fundamental group in a point $\Gamma g_0$ with the group $\Gamma$ of deck transformations as in Section \ref{intro}. Since $\alpha: G \to G$ is a diffeomorphism, we know that $\alpha_\ast(\gamma) = \alpha \gamma \alpha^{-1}$ for all $\gamma \in \Gamma$. If we write $\alpha = (g,\delta)$, then $$\alpha N \alpha^{-1} = g \delta(N) g^{-1}  \subseteq \Gamma \cap G = N$$ and thus $\alpha$ induces an affine nilmanifold endomorphism on the finite covering space. The holonomy group $\faktor{\Gamma}{N} = F$ is isomorphic to a subgroup of $\Aut(G)$ in a natural way. The group morphism induced by $\alpha_\ast$ on $F$ is given by $\delta f \delta^{-1}$ for all $f \in F$ under this isomorphism. Therefore this map is injective and thus the statement of the theorem follows from Proposition \ref{cover}.
\end{proof}

\noindent This theorem reduces the study of periodic points of affine infra-nilmanifold endomorphisms to the study of affine nilmanifold endomorphisms. Since we are interested in the density of (eventually) periodic points, the following lemma is important.

\begin{Lem}
\label{density}
Let $p: E \to B$ a covering map and consider subsets $X \subseteq E$ and $Y \subseteq B$. Then the following statements are true.
\begin{enumerate}[$(1)$]
\item If $X$ is dense in $E$, then $p(X)$ is dense in $B$.
\item $Y$ is dense in $B$ if and only if $p^{-1}(Y)$ is dense in $E$.
\end{enumerate}
\end{Lem}
\noindent Note that it is not true that $X$ is dense in $E$ if $p(X)$ is dense in $B$. For example, if $p: \R \to \Z \backslash \R$ is the natural covering map, then $p([0,1])$ is dense in $\Z \backslash \R$ but $[0,1]$ is not dense in $\R$.

\begin{proof}
The first statement is true since $p$ is a continuous surjective map. For the second statement, take an open set $U \subseteq E$ and consider the subset $p(U) \subseteq B$ which is also open since $p$ is a local homeomorphism. Since $Y$ is dense, there exists $y \in Y \cap p(U)$ and take $u \in U$ such that $p(u) = y$. Because $u \in p^{-1}(Y) \cap U$, this finishes the proof.
\end{proof} 

Not every affine infra-nilmanifold endomorphism has periodic points, as we can see from the following example.

\begin{Ex}
\label{noperiod}
Consider the torus $\Z^n \backslash \R^n$ and the affine infra-nilmanifold endomorphism induced by the affine map $\alpha = (b, I_n) \in \Aff(\R^n)$. Note that $$\bar{\alpha}^k (\Z^n + x) = \Z^n + x + k b$$ and thus $\Z^n + x$ is a periodic point if and only if $k b \in \Z^n$ for some $k > 0$. This implies that the map $\bar{\alpha}$ has periodic points if and only if $b \in \Q^n$. By taking $b \in \R^n \setminus \Q^n$, this gives us a class of affine infra-nilmanifold endomorphisms without periodic points.
\end{Ex}

Example \ref{noperiod} shows that there are affine infra-nilmanifold endomorphisms which have no periodic points at all. For studying periodic points, the following definition is important.

\begin{Def}
Let $f: M \to M$ and $g: N \to N$ be maps on the manifolds $M$ and $N$. We call $f$ and $g$ \textbf{topologically conjugate} if there exists a homeomorphism $h: M \to N$ which makes the following diagram commutative:
$$ \xymatrix{
M \ar[r]^h \ar[d]^f & N \ar[d]^g \\
M \ar[r]^h 			& N.} $$ 
The homeomorphism $h$ is called a \textbf{topologically conjugation} between $f$ and $g$. 
\end{Def}

Note that topological conjugacy is an equivalence relation on all self-maps of manifolds. If two maps are topologically conjugate, there is an immediate relation between the sets of (eventually) periodic points.

\begin{Prop}
\label{basictopcon}
Let $M$ and $N$ be manifolds and assume that $h: M \to N$ is a topologically conjugacy between maps $f: M \to M$ and $g: N \to N$. Then $h(\Fix(f)) = \Fix(g)$, $h(\Per(f)) = \Per(g)$ and $h(\ePer(f)) = h(\ePer(g))$.
\end{Prop}
\noindent This proof is immediate and therefore we omit it.

\smallskip

In Example \ref{noperiod} we showed that there exist affine infra-nilmanifold endomorphisms which have no periodic points at all. If an affine nilmanifold endomorphism has a periodic point, we can further reduce the problem to nilmanifold endomorphisms by the following result.

\begin{Thm}
\label{topcon}
Let $\bar{\alpha}: \Gamma \backslash G \to \Gamma \backslash G$ be an affine infra-nilmanifold endomorphism induced by the affine transformation $\alpha \in \Aff(G)$. Then $\bar{\alpha}$ is topologically conjugate to an infra-nilmanifold endomorphism if and only if $\bar{\alpha}$ has a fixed point.
\end{Thm}

\noindent The proof of Theorem \ref{topcon} is identical to the proof of \cite[Theorem 4.5.]{deki11-1} which states that every expanding affine infra-nilmanifold endomorphisms is topologically conjugate to an expanding infra-nilmanifold endomorphism. We do give the proof here, since it also gives an exact form of the topologically conjugacy between these maps.
\begin{proof}
One direction of Theorem \ref{topcon} is clear, since every nilmanifold endomorphism has a fixed point and a topological conjugation maps fixed points to fixed points by Proposition \ref{basictopcon}.

For the other direction, write $\alpha = (g,\delta)$ and let $\Gamma g_0$ be a fixed point of $\bar{\alpha}$. Identify the Lie group $G$ with the subgroup of pure translations in $\Aff(G)$. Consider the subgroup $\Gamma^{\prime} = g_0^{-1} \Gamma g_0 \le \Aff(G)$ which is also an almost-Bieberbach group modeled on the Lie group $G$. The point $\Gamma g_0$ is a fixed point, so $ \delta(g_0) \in g^{-1}\Gamma g_0$ or equivalently $g_0^{-1} g \delta(g_0) \in \Gamma^\prime$. Since $\alpha \Gamma \alpha^{-1} = g \delta \Gamma \delta^{-1} g^{-1} \le \Gamma$, we have that $\delta \Gamma \delta^{-1}\le g^{-1} \Gamma g$.

Consider the homeomorphism $h: \Gamma \backslash G \to \Gamma^\prime \backslash G$ given by $$h(\Gamma g) = \Gamma^\prime g_0^{-1} g.$$ This map is indeed well-defined, since if $\gamma \in \Gamma$, then $\gamma^\prime = g_0^{-1} \gamma g_0 \in \Gamma^\prime$ and we have that $$g_0^{-1} \gamma g = \gamma^\prime g_0^{-1} g \in \Gamma^\prime g_0^{-1} g.$$ The automorphism $\delta$ induces an infra-nilmanifold endomorphism on $\Gamma^\prime \backslash G$ since $$\delta \Gamma^\prime \delta^{-1} = \delta(g_0^{-1}) \delta \Gamma \delta^{-1} \delta(g_0) = g_0^{-1} \Gamma g g^{-1} \Gamma g g^{-1} \Gamma g_0 \le g_0^{-1} \Gamma g_0. $$ Moreover, for every $\Gamma^\prime x \in \Gamma^\prime \backslash G$, we have that $$h \left( \bar{\alpha} \left( h^{-1}(\Gamma^\prime x)\right)\right) = h (\bar{\alpha}(\Gamma g_0 x)) = h( \Gamma g \delta(g_0 x)) = \Gamma^\prime g_0^{-1} g \delta(g_0) \delta(x) = \Gamma^\prime \delta(x)  $$ and so $h \bar{\alpha} h^{-1} = \bar{\delta}$. Thus $h$ is a topological conjugacy between $\bar{\alpha}$ and $\bar{\delta}$.
\end{proof}

\begin{Rmk}
\label{explicitform}
The proof of Theorem \ref{topcon} also gives us the exact form of a possible topological conjugation between $\bar{\alpha} $ and $\bar{\delta}$. For the remaining part of this paper, we are mainly interested in the lift of this homeomorphism to the universal cover $G$ and the subgroup $N$ of pure translations. 

If $\Gamma g_0$ is the fixed point of $\bar{\alpha}$, then there is a topological conjugacy which maps $\Gamma g_0$ to $\Gamma^{\prime} e$ with lift to the universal cover $G$ given by $L_{g_0^{-1}}$. If $N$ and $N^\prime$ are the subgroups of pure translations of $\Gamma$ and $\Gamma^\prime$ respectively, then $N^\prime = g_0^{-1} N g_0$ and thus $$\left(N^\prime\right)^\Q = g_0^{-1} N^\Q g_0.$$
\end{Rmk}

\noindent In the following section, we compute $\Per(\bar{\delta})$ and $\ePer(\bar{\delta})$ for nilmanifold endomorphisms $\bar{\delta}$. The results of this section thus allow us to interpret these results in terms of $\Per(\bar{\alpha})$ and $\ePer(\bar{\alpha})$ from general affine infra-nilmanifold endomorphisms.

\section{Sufficient condition for (eventually) periodic points}
In this section, we give a sufficient condition for points of a nilmanifold to be periodic, based on the relative order of an element in the radicable hull.  We first define this relative order and show how it is related to the index of a subgroup for nilpotent groups.

\begin{Def}Let $G$ be a group and take $H \le G$ any subgroup of finite index. We define the \textbf{relative order} of an element $g \in G$ with respect to the subgroup $H$ as $$\ord(g) = \min \{ n > 0 \mid g^n \in H\}.$$
\end{Def} Note that $\ord(g)$ does not depend on the overlying group $G$, therefore we avoid using the group $G$ in the notation. If $N$ is a normal subgroup with natural projection map $\pi: G \to \faktor{G}{N}$, then $\ordn(g)$ is the order of the element $\pi(g)$ in $\faktor{G}{N}$. From the definition, it follows that $g^n \in H$ if and only if $\ord(g) \mid n$. Let $H_1 \le H_2 \le G$ be subgroups, then by definition $g^{\ordz_{H_1}(g)} \in H_1 \le H_2$ and thus $\ordz_{H_2}(g) \mid \ordz_{H_1}(g)$. 

Let $s > 0$ be any integer and consider the set $$X_s = \left\{ g \in G \suchthat \ord(g) \mid s \right\}.$$ The set $X_s$ does not form a subgroup of $G$ in general and denote by $H^{\frac{1}{s}}$ the subgroup generated by $X_s$. If $\varphi: G \to G$ is a group morphism such that $\varphi(H) \le H$, then $$\varphi(g)^{\ord(g)} = \varphi(g^{\ord(g)}) \in \varphi(H) \le H$$ and thus $\ord(\varphi(g)) \mid \ord(g)$. This implies that $\varphi(X_s) \subseteq X_s$ and thus $\varphi$ induces a group morphism on every subgroup $H^{\frac{1}{s}}$ for $s > 0$. 

If $N$ is a $\mathcal{F}$-group with radicable hull $N^\Q$, then every element $n \in N^\Q$ lies in a subgroup which has $N$ as a finite index subgroup. Therefore, we can define $\ordn(n)$ for every element $n \in N^\Q$. The group $N$ is always a subgroup of finite index of the group $N^{\frac{1}{s}}$ in this case. By using this relative order on $N^\Q$, we can construct eventually periodic points for every nilmanifold endomorphism.
\begin{Thm}
\label{subeper}
Let $N$ be a lattice of the nilpotent Lie group $G$ with radicable hull $N^\Q$. If $\bar{\delta}: N \backslash G \to N \backslash G$ is a nilmanifold endomorphism then $$p(N^\Q) \subseteq \ePer(\bar{\delta}).$$
\end{Thm}

\begin{proof}
Take $n \in N^\Q$ arbitrary and let $s = \ordn(n)$, which implies that $n \in N^{\frac{1}{s}}$. Since $N$ is a finite index subgroup of $ N^{\frac{1}{s}}$, the set $p(N^{\frac{1}{s}})$ is finite. From the previous discussion we know that $\delta(N^{\frac{1}{s}}) \le N^{\frac{1}{s}}$ because  $\delta(N) \le N$. This implies that $\bar{\delta}$ induces a map on the finite set $p(N^{\frac{1}{s}})$ and thus every point of $p(N^{\frac{1}{s}})$ is eventually periodic. In particular, $p(n) \in p(N^{\frac{1}{s}})$ is eventually periodic.
\end{proof}

To construct periodic points, we need the additional property that the map $\bar{\delta}$ is injective on the set $p\left(N^{\frac{1}{s}}\right)$. Therefore we need a more careful study of the groups $N^{\frac{1}{s}}$. A first step is the relation between the index of a subgroup and the relative order.

For normal subgroups $N$ of $G$, it follows from Cauchy's theorem that if $p \mid [G:N]$, then there always exists an element $g \in G$ such that $\ordn(G) = p$. We show that this is also true for subnormal subgroups $H$ of finite index. Recall that a subgroup $H \le G$ is called subnormal if there exists a sequence of subgroups $$H = H_0 \le H_1 \le \ldots \le H_m = G$$ with $H_j$ a normal subgroup of $H_{j+1}$. 
\begin{Lem}
\label{cauchy}
Let $G$ be a group and $H \le G$ a subnormal subgroup of finite index. If $p \mid [G:H]$, then there exists $g \in G$ such that $$\ord(g) = p.$$
\end{Lem} 
\noindent In particular, we can apply this lemma for every nilpotent group $G$, since all subgroups of a nilpotent group are subnormal.
\begin{proof}
The subgroup $H$ is subnormal in $G$, meaning that there exists a sequence of subgroups $$H = H_0 \le H_1 \le \ldots \le H_m = G$$ such that $H_i$ is normal in $H_{i+1}$. Since $p \mid [G:H]$, there exists an $i$ such that $ p \mid [H_{i+1}:H_i]$ and by Cauchy's theorem we know that there exists $g \in H_{i+1}$ such that $p = \ordz_{H_i}(g)$. Because $H \le H_i$, we have that $p = \ordz_{H_i}(g) \mid \ord(g)$ and thus a power of the element $g$ satisfies the condition of the lemma.
\end{proof}

So if $H \le G$ is a finite index subgroup, then the relative order of an element $g \in G$ gives us information about the index of the subgroup $H$ in $G$. The following result of \cite[Proposition 6.3]{sega83-1} shows how the relative order $\ordn(n)$ of elements $n \in N^{\frac{1}{s}}$ looks like.

\begin{Prop}
\label{index2}
Let $G$ be a nilpotent group of nilpotency class at most $c$ and take any $s \in \N$. Let $X$ be a set of generators of $G$ and put $H = \langle x^s \mid x \in X\rangle$. Then $$G^{s^m} \le H$$ for $m = \frac{1}{2} c (c+1)$.
\end{Prop}
\noindent Note that the result can be translated to relative orders of subgroups, since $$G^{s^m} \le H$$ is equivalent to $$\ord(g) \mid s^m$$ for all $g \in G$. If we apply this result to the groups $H^{\frac{1}{s}}$, we get the following translation.

\begin{Prop}
\label{stothek}
Let $G$ be a finitely generated nilpotent group and $H \le G$ a subgroup. Then $[H^{\frac{1}{s}}:H] \mid s^k$ for some $k \in \N$. 
\end{Prop}

\begin{proof}
Apply Proposition \ref{index2} to the set of generators $$X_s = \{g \in G \mid \ord(g) \mid s\}$$ of $H^{\frac{1}{s}}$. The group generated by the elements $x^s$ for $x \in X_s$ is a subgroup of $H$, which we denote by $H^\prime$. Suppose that $[H^{\frac{1}{s}}:H] \nmid s^k$ for every $k > 0$ or equivalently that there exists a prime $p$ with $\gcd(p,s) = 1$ and $p \mid [H^{\frac{1}{s}}:H]$. Then also $p \mid [H^{\frac{1}{s}}:H^\prime]$ and because of Lemma \ref{cauchy} we get that there exists a $g \in H^{\frac{1}{s}}$ such that $\ordz_{H^\prime}(g) = p$. This is a contradiction to Proposition \ref{index2}.
\end{proof}

To construct periodic points for a nilmanifold endomorphism, we start from the idea of Theorem \ref{subeper} in which we constructed eventually periodic points. Following Lemma \ref{elem}, the only thing we still need to find periodic points is injectivity of the induced map. This idea is exploited in the following theorem by restricting to certain groups $N^{\frac{1}{s}}$ depending on the determinant of the endomorphism.

\begin{Thm}
\label{perD}
Let $N \backslash G$ be a nilmanifold. Consider the radicable hull $N^\Q$ of $N$ and the subset $$N^\Q_{D} = \big\{n \in N^\Q \suchthat \gcd(D,\ordn(g))=1\big\}.$$ If $\bar{\delta}$ is a nilmanifold endomorphism on $N \backslash G$ with determinant $D$, then $$p(N_D^\Q) \subseteq \Per(\bar{\delta}).$$
\end{Thm}

\begin{proof}
Take $n \in N^\Q_D$, so $\ordn(n) = s$ with $\gcd(D,s) = 1$. Consider the group $N^{\frac{1}{s}}$ which contains $N$ as a subgroup of finite index and note that $n \in N^{\frac{1}{s}}$. We claim that $\bar{\delta}$ is injective on the finite set $p(N^{\frac{1}{s}})$. Because of Lemma \ref{elem}, this implies that the point $p(n)$ is periodic.

Injectivity of $\bar{\delta}$ on $p(N^{\frac{1}{s}})$ is equivalent to showing that $\delta(N^{\frac{1}{s}}) \cap N = \delta(N)$. Indeed, assume that $\bar{\delta}(p(n_1)) = N \delta(n_1)  = N \delta(n_2) = \bar{\delta}(p(n_2))$ for $n_i \in N^{\frac{1}{s}}$, then $$\delta(n_1) \delta(n_2^{-1}) = \delta(n_1n_2^{-1})\in N \cap \delta(N^{\frac{1}{s}}) = \delta(N)$$ and therefore $n_1 n_2^{-1}  \in N$ because $\delta$ is injective. So this implies $N n_1  = N n_2$ and thus that $\bar{\delta}$ is injective on $p(N^{\frac{1}{s}})$.

So the only thing left to show is that $\delta(N^{\frac{1}{s}}) \cap N = \delta(N)$. We know that $\delta(N) \le \delta(N^{\frac{1}{s}}) \cap N $ and that $\delta(N)$ is a subgroup of index $\vert D \vert$ in $N$. Therefore it suffices to show that also $ \delta(N^{\frac{1}{s}}) \cap N$ is a subgroup of index $\vert D \vert$ in $N$ to conclude that $\delta(N) = \delta(N^{\frac{1}{s}}) \cap N$. The subgroup $\delta(N^{\frac{1}{s}})$ is a subgroup of index $\vert D \vert$ in $N^{\frac{1}{s}}$. It follows that $$\vert D \vert = [N^{\frac{1}{s}}: \delta(N^{\frac{1}{s}})]\hspace{1mm} \big| \hspace{1mm} [N^{\frac{1}{s}}: \delta(N^{\frac{1}{s}}) \cap N ] = [N^{\frac{1}{s}}:N] [N:\delta(N^{\frac{1}{s}}) \cap N]$$ where $[N^{\frac{1}{s}}:N] \hspace{1mm} \big| \hspace{1mm} s^k$ for some $k > 0$ by Proposition \ref{stothek}. Since $\gcd(s,D) = 1$ this implies that $\delta(N^{\frac{1}{s}}) \cap N $ has index $\vert D \vert$ in $N$. Thus $\delta(N^{\frac{1}{s}}) \cap N = \delta(N)$ which finishes the proof.
\end{proof}

The following result is an immediate corollary of Theorem \ref{perD}.

\begin{Thm}
\label{maindense}
Let $\bar{\alpha}$ be an affine infra-nilmanifold endomorphism of the infra-nilmanifold $\Gamma \backslash G$. Then either $\Per(\bar{\alpha}) = \emptyset$ or $\Per\left(\bar{\alpha}\right)$ is a dense subset of $\Gamma \backslash G$.
\end{Thm}

\begin{proof}
Assume that $\bar{\alpha}$ has periodic points, then we show that $\Per(\bar{\alpha})$ is dense. Since $\Per(\bar{\alpha}^k) = \Per(\bar{\alpha})$ we can take some power of $\bar{\alpha}$ and assume that $\bar{\alpha}$ has a fixed point. By Theorem \ref{topcon} we can therefore assume that $\bar{\alpha}$ is an infra-nilmanifold endomorphism. An application of Theorem \ref{reduc} and Lemma \ref{density} shows that it suffices to consider the case of a nilmanifold endomorphism. 

So let $\bar{\delta}: N \backslash G \to N \backslash G$ be a nilmanifold endomorphism. Denote by $p: G \to N \backslash G$ the projection map and let $D$ be the determinant of $\delta \in \Aut(G)$. The set $N_D^\Q$ forms a dense subset of $G$ and Lemma \ref{density} implies that $p(N_D^\Q)$ is dense in $\Gamma \backslash G$. Since Theorem \ref{perD} shows that $p(N_D^\Q)$ is a subset of $\Per(\bar{\delta})$, this ends the proof.
\end{proof}

%
%
%

%

In this section we constructed a dense subset of periodic points for affine infra-nilmanifold endomorphisms $\bar{\alpha}$. A more general question is to give a full description of the sets $\Per(\bar{\alpha})$ and $\ePer(\bar{\alpha})$. For this we first need a necessary condition for a point $\Gamma g$ of $\Gamma \backslash G$ to be (eventually) periodic, which we deduce in the following section.

\section{Necessary condition for (eventually) periodic points}

In this section, we give a necessary condition for a point to be a(n) (eventually) periodic point of a nilmanifold endomorphism. Combined with the results of the previous sections, this will give us a complete description of the set of eventually periodic points, also in the case of general affine infra-nilmanifold endomorphisms.

Every torsion-free, radicable nilpotent group $N^\Q$ is the radicable hull of an $\mathcal{F}$-group $N$ and every such a group $N^\Q$ corresponds to a rational Lie algebras $\lien^\Q$ via the Baker-Campbell-Hausdorff formula. Under this correspondence, group morphisms correspond to Lie algebra automorphisms and radicable subgroups $H^\Q$ of $N^\Q$ correspond to rational subalgebras of $\lien^\Q$. If $\varphi: N^\Q \to N^\Q$ is a group morphism, then $\ker(\varphi)$ is a radicable normal subgroup of $N^\Q$. By $N^\R$ we will denote the unique simply connected and connected Lie group which contains $N^\Q$ as a dense subgroup and this Lie group corresponds to the real Lie algebra $\R \otimes \lien^\Q = \lien^\R$. The group $N$ is a lattice of the Lie group $N^\R$. Every group morphism $\varphi: N^\Q \to N^\Q$ uniquely extends to an element of $\Aut(N^\R)$ and will denote this extension as $\varphi^\R$. For more details, we refer to \cite{sega83-1}.

Let $G$ be a simply connected and connected nilpotent Lie group with lattice $N \le G$, implying that $G \approx N^\R$. Assume that $\bar{\delta}$ is a nilmanifold endomorphism induced by $\delta \in \Aut(G)$. If the coset $N g$ is an eventually periodic points of $\bar{\delta}$, then there exist $0 \leqslant k_1 < k_2$ such that $$N \delta^{k_1}(g) = N \delta^{k_2}(g)$$ or equivalently such that $$\delta^{k_2}(g) \left(\delta^{k_1}(g)\right)^{-1} = \delta^{k_2}(g) \delta^{k_1}\left(g^{-1}\right) \in N \le N^\Q.$$ More generally, we are interested in the question which elements $n$ of the real Mal'cev completion $N^\R$ of a torsion-free radicable nilpotent group $N^\Q$ satisfy a relation of the form 
\begin{align*}
\varphi^\R(n) \psi^\R\left(n^{-1}\right) \in N^\Q
\end{align*}
with $\varphi, \psi: N^\Q \to N^\Q$ automorphisms of the group $N^\Q$. The following theorem gives an answer to this question.
\begin{Thm}
\label{coset}
Let $N^\Q$ be a torsion-free radicable nilpotent group with group morphisms $\varphi, \psi: N^\Q \to N^\Q$ and take $H^\Q \le N^\Q$ the radicable subgroup defined as $$H^\Q = \left\{h \in N^\Q \suchthat \varphi(h) = \psi(h)\right\}.$$ Then for all $n \in N^\R$ we have the following equivalence:
\begin{align*}
\varphi^\R(n) \psi^\R(n)^{-1} \in N^\Q \iff n \in N^\Q H^\R.
\end{align*}
\end{Thm}

\noindent In the statement of Theorem \ref{coset}, $N^\Q H^\R$ is the product of the subgroups $N^\Q$ and $H^\R$ which in general does not form a subgroup of $N^\R$. Note that the group $H^\Q$ is indeed a radicable subgroup, since if $h \in H^\Q$, then 
$$\varphi(h^{\frac{1}{s}}) = \varphi(h)^{\frac{1}{s}} = \psi(h)^{\frac{1}{s}} = \psi(h^{\frac{1}{s}})$$ and thus $h^{\frac{1}{s}} \in H^\Q$. To prove Theorem \ref{coset}, we will use induction depending on the lower central series of $N^\Q$, combined with the following easy lemma, which corresponds to the abelian case of Theorem \ref{coset}.

\begin{Lem}
\label{image}
Let $\varphi: \Q^n \to \Q^m$ be a linear map and take the unique extension $\varphi^\R: \R^n \to \R^m$. Then it holds that $\varphi^\R(\R^n) \cap \Q^m = \varphi(\Q^n)$.
\end{Lem}

\begin{proof}
Consider the rational subspace $V^\Q = \ker(\varphi) \subseteq \Q^n$. Take a complementary subspace $W^\Q \subseteq \Q^n$ for $V^\Q$, meaning that $\Q^n$ can be written as the direct sum $\Q^n = V^\Q \oplus W^\Q$. The restriction $\left. \varphi \right|_{W^\Q}$ of the linear map $\varphi$ to the subspace $W^\Q$ forms an isomorphism between $W^\Q$ and $\varphi(\Q^n)$. Since $\R$ can also be written as the direct sum $\R^n = V^\R \oplus W^\R = \ker(\varphi^\R) \oplus W^\R$, $\left. \varphi^\R \right|_{W^\R}$ forms an isomorphism between $W^\R$ and $\varphi^\R(\R^n)$ which is the extension of $\left. \varphi \right|_{W^\Q}$. This implies that $\varphi^\R(\R^n) \cap \Q^m = \varphi(\Q^n)$.
\end{proof}

The abelian case, so for linear maps $\varphi, \psi: \Q^n \to \Q^n$, follows from applying Lemma \ref{image} to the linear map $\varphi - \psi: \Q^n \to \Q^n$. By using induction, we can now prove the general statement of Proposition \ref{coset}.

\begin{proof}[Proof of Theorem \ref{coset}] One implication is immediate, namely if $n \in N^\Q H^\R$, so $n = m h$ for some $m \in N^\Q, h \in H^\R$, then\begin{align*} \varphi(n) \psi(n)^{-1} &= \varphi(mh) \psi(mh)^{-1}\\&= \varphi(m) \varphi(h) \psi(h)^{-1} \psi(m)^{-1}\\ &= \varphi(m) \psi(m)^{-1} \in N^\Q.\end{align*} 

For the other implication we first introduce some notations. Consider the abelian groups $\faktor{\gamma_i(N^\Q)}{\gamma_{i+1}(N^\Q)}$ with quotient maps $\pi_i: \gamma_i(N^\Q) \to \faktor{\gamma_i(N^\Q)}{\gamma_{i+1}(N^\Q)}$. Take $N^\Q_1$ equal to the group $N^\Q$ and consider $\alpha_1: N^\Q_1 \to \faktor{N^\Q}{\gamma_2(N^\Q)}$ the map defined by $$\alpha_1(n) = \pi_1 \left(\varphi(n) \psi(n)^{-1} \right) = \pi_1(\varphi(n)) - \pi_1(\psi(n)).$$ The map $\alpha_1$ is a group morphism since for every $m, n  \in N^\Q_1$, we have \begin{align*} \alpha_1(m n) &= \pi_1(\varphi(m n)) - \pi_1(\psi( m n)) \\ &=  \pi_1(\varphi(m)) - \pi_1(\psi(m)) + \pi_1(\varphi(n)) - \pi_1(\psi( n)) \\ &= \alpha_1(m) + \alpha_1(n).\end{align*}  Denote by $N^\Q_2$ the kernel of the group morphism $\alpha_1$. 

Inductively, we also define the group morphism $\alpha_i: N^\Q_i \to \faktor{\gamma_i(N^\Q)}{\gamma_{i+1}(N^\Q)}$ given by $$\alpha_i(n) = \pi_i(\varphi(n) \psi(n)^{-1})$$ and the subgroup $N^\Q_i$ of $N^\Q$ as the kernel of the morphism $\alpha_{i-1}$. To show that these maps $\alpha_i$ are indeed group morphisms, we first consider the map $\tilde{\alpha}_i: N^\Q_i \to \faktor{N^\Q}{\gamma_{i+1}(N^\Q)}$ which is the composition of $\alpha_i$ and the natural inclusion $\faktor{\gamma_i(N^\Q)}{\gamma_{i+1}(N^\Q)}$ into $\faktor{N^\Q}{\gamma_{i+1}(N^\Q)}$. A computation shows that 
\begin{align*}\tilde{\alpha}_i(m n) &=\varphi(mn) \psi(mn)^{-1} \gamma_{i+1}(N^\Q) \\ &= \varphi(m) \varphi(n) \psi(n)^{-1} \psi(m)^{-1} \gamma_{i+1}(N^\Q)\\ &= \varphi(m) \tilde{\alpha}_i(n) \psi(m)^{-1} \gamma_{i+1}(N^\Q) \\ &= \tilde{\alpha}_i(m) \tilde{\alpha}_i(n) [\tilde{\alpha}_i(n), \psi(m)^{-1}] \gamma_{i+1}(N^\Q) \\ &= \tilde{\alpha}_i(m) \tilde{\alpha}_i(n) \end{align*} where the last equality holds because $\tilde{\alpha}_i(n) \in \faktor{\gamma_i(N^\Q)}{\gamma_{i+1}(N^\Q)}$. This shows that the map $\tilde{\alpha}_i$ is a group morphism and therefore also the map $\alpha_i$.

The groups $N_i^\Q$ are radicable subgroups of $N^\Q$ and thus we can consider the groups $N_i^\R$. The group $H^\Q$ is a subgroup of every group $N_i^\Q$ since for every $h \in H^\Q$ we have $\varphi(h) \psi(h)^{-1} = 0$. Every group $N_{i+1}^\Q$ is a normal subgroup of $N_i^\Q$ and we have group morphisms $$\bar{\alpha}_i: \faktor{N_i^\Q}{N_{i+1}^\Q} \to \faktor{\gamma_i(N^\Q)}{\gamma_{i+1}(N^\Q)}.$$ The extension $\bar{\alpha}_i^\R$ is given by $$\bar{\alpha}_i^\R: \faktor{N_i^\R}{N_{i+1}^\R} \to \faktor{\gamma_i(N^\R)}{\gamma_{i+1}(N^\R)}$$ with $\bar{\alpha}_i^\R(n) = \pi_i^\R(\varphi^\R(n) \psi^\R(n)^{-1})$. Since $N^\Q$ is $c$-step nilpotent, the group $\faktor{\gamma_{i}(N^\Q)}{\gamma_{i+1}(N^\Q)} = 1$ for every $i \geq c+1$ and therefore the groups $N_i^\Q \le H^\Q$ for every $i \geq c+1$. In particular we get $H = N_{c+1}^\Q$. 

With these groups $N_i^\Q$, we can prove the other implication by using induction on the subgroups $N_i^\R$. If $n \in N_{c+1}^\R$, there is nothing to prove since $N_{c+1}^\R = H^\R$. Now assume that the proposition is true for every $n \in N_{i+1}^\R$, then we will show it is true for $n \in N_{i}^\R$. Consider the map $\bar{\alpha}_i: \faktor{N_i^\Q}{N_{i+1}^\Q} \to \faktor{\gamma_i(N^\Q)}{\gamma_{i+1}(N^\Q)}$ with extension $\bar{\alpha}_i^\R: \faktor{N_i^\R}{N_{i+1}^\R} \to \faktor{\gamma_i(N^\R)}{\gamma_{i+1}(N^\R)}$. This map satisfies the properties of Lemma \ref{image}, so there exists $m \in N_i^\Q$ such that $\bar{\alpha}_i^\R(m) = \bar{\alpha}_i^\R(n)$ or thus that $\tilde{n} = m^{-1} n \in N_{i+1}^\R$. Since $m \in N^\Q$ and thus also $\varphi(m^{-1}), \psi(m) \in N^\Q$, we have that $$\varphi(\tilde{n}) \psi(\tilde{n}^{-1}) = \varphi(m^{-1}) \underbrace{\varphi(n) \psi(n^{-1})}_{\in N^\Q} \psi(m) \in N^\Q.$$ By induction we conclude that $n \in N^\Q H^\R$. 
\end{proof}

As discussed above Theorem \ref{coset}, this results makes it possible to study (eventually) periodic points. A combination with Theorem \ref{subeper} gives us the following result.

\begin{Thm}
\label{fulldes}
Let $\bar{\delta}: N \backslash G \to N \backslash G$ be a nilmanifold endomorphism induced by the map $\delta \in \Aut(G)$ on the nilmanifold $ N \backslash G$ with projection map $p: G \to N \backslash G$. Let $H^\R$ be the subgroup of $G$ defined as $$H^\R = \left\{h \in G \suchthat \exists k > 0:  \delta^k (h) =h\right\},$$ then $$\ePer(\bar{\delta}) = p(N^\Q H^\R).$$
\end{Thm}

\noindent The subgroup $H^\R$ of $G$ in this theorem corresponds to the Lie algebra given by the eigenspaces of $\delta$ corresponding to roots of unity. This subgroup can thus be easily computed starting from the automorphism $\delta$.

\begin{proof}
From the discussion above Theorem \ref{coset}, we know that an eventually periodic point $p(g) = Ng$ has to satisfy $\delta^{k_2}(g) \delta^{k_1}\left(g^{-1}\right) \in N^\Q$ for some $0 < k_1 < k_2$. Since $\delta^{k_1}: N^\Q \to N^\Q$ is invertible, this is equivalent to $\delta^k(g) g^{-1} \in N^\Q$ for some $k >0$. Theorem \ref{coset} applied to $\delta^k$ and $\I_{N^\Q}$ implies that $g$ is indeed an element of $N^\Q H^\R$, showing that the inclusion $\ePer(\bar{\delta}) \subseteq p(N^\Q H^\R)$ holds. 

Thus it suffices to show that every point of $p(N^\Q H^\R)$ is eventually periodic. By replacing $\delta$ by a power of itself, we can assume that $h \in H^\R$ implies $\delta(h) = h$. Let $g = n h$ with $n \in N^\Q$ and $h \in H^\R$, then Theorem \ref{subeper} implies that $p(n)$ is eventually periodic or thus that there exist distinct $k_1, k_2 > 0$ such that $$\bar{\delta}^{k_1}(p(n)) = \bar{\delta}^{k_2}(p(n)).$$ Now 
\begin{align*}
\bar{\delta}^{k_1}(p(g)) = N \delta^{k_1}(n) \delta^{k_1}(h) = N \delta^{k_2}(n) h = N \delta^{k_2}(n) \delta^{k_2}(h) = \bar{\delta}^{k_2}(p(g))
\end{align*}
and thus $p(g)$ is eventually periodic.
\end{proof}

By combining Theorem \ref{fulldes} with Theorem \ref{topcon} (and the explicit form of the homeomorphism as described in Remark \ref{explicitform}), we get the following corollary for affine infra-nilmanifold endomorphisms. 

\begin{Cor}
Let $\bar{\alpha}: \Gamma \backslash G \to \Gamma \backslash G$ be an affine infra-nilmanifold endomorphism induced by the affine transformation $\alpha = (g, \delta) \in \Aff(G)$. Denote by $p: G \to \Gamma \backslash G$ the projection map and by $H^\R$ the subgroup of $G$ defined as $$H^\R =\{h \in G \mid \exists k > 0:  \delta^k (h) =h\}.$$ If $\bar{\alpha}$ has a periodic point $\Gamma g_0$, then $$\ePer (\bar{\alpha}) = p( N^\Q g_0 H^\R).$$
\end{Cor}
Note that $N^\Q$ in the statement is the radicable hull of the Fitting subgroup of $\Gamma$, which is in general different from the radicable hull of the infra-nilmanifold endomorphism $\bar{\delta}$ constructed by Theorem \ref{topcon}.

\smallskip
For periodic points there is a similar description. The only problem is that we do not know which points of $p(N^\Q)$ are periodic, since we only constructed a subset in Theorem \ref{perD}.

\begin{Thm}
\label{split}
Let $\bar{\delta}: N \backslash G \to N \backslash G$ be a nilmanifold endomorphism induced by the map $\delta \in \Aut(G)$ on the nilmanifold $ N \backslash G$ with projection map $p: G \to N \backslash G$. Let $H^\R$ be the subgroup of $G$ defined as $$H^\R = \{h \in G \mid \exists k > 0:  \delta^k (h) =h\}$$ and take $X = p(N^\Q) \cap \Per(\bar{\delta})$. Then $$\Per(\bar{\delta}) = p( p^{-1} \left(X \right) H^\R).$$
\end{Thm}

\begin{proof}
By taking some power of $\delta$, we can assume that $\delta(h) = h$ for all $h \in H^\R$. Let $p(g)$ be a periodic point of $\bar{\delta}$. Since every periodic point is eventually periodic, we can assume that $g$ is of the form $g = n h$ with $ n \in N^\Q$ and $h \in H^\R$ by Theorem \ref{fulldes}. If $p(g)$ is periodic, we know that $\delta^k(g) g^{-1} \in N$ for some $k > 0$. Since $$  \delta^k(g) g^{-1} = \delta^k(n) \delta^k(h) h^{-1} n^{-1} = \delta^k(n) n^{-1} \in N$$ this implies that $p(n)$ is a periodic point, so $p(n) \in X$. We conclude that $\Per(\bar{\delta}) \subseteq p( p^{-1} \left(X \right) H^\R)$. 

The other inclusion is identical as in Theorem \ref{fulldes}.
\end{proof}

Again, a combination of Theorem \ref{split} and Theorem \ref{topcon} gives us the following result.

\begin{Cor}
\label{fulldesper}
Let $\bar{\alpha}: \Gamma \backslash G \to \Gamma \backslash G$ be an affine infra-nilmanifold endomorphism induced by the affine transformation $\alpha = (g, \delta) \in \Aff(G)$ and assume that $\bar{\alpha}$ has a periodic point $\Gamma g_0$. Denote by $p: G \to \Gamma \backslash G$ the projection map and by $H^\R$ the subgroup of $G$ defined as $$H^\R =\{h \in G \mid \exists k > 0:  \delta^k (h) =h\}$$ and take $X = p(N^\Q) \cap \Per(\bar{\delta})$ where $\bar{\delta}$ is the nilmanifold endomorphism induced by $\delta$ on $g_0^{-1} N g_0$. The periodic points of $\bar{\alpha}$ are equal to $$\Per (\bar{\alpha}) = p(p^{-1}(X) g_0  H^\R).$$
\end{Cor}

Theorem \ref{split} shows us that for a full description of $\Per(\bar{\alpha})$, we only need to know how the set $X = \Per(\bar{\delta}) \cap p(N^\Q)$ looks like. In Section \ref{periodexam} we give some examples showing that this set can be quite wild. The only information we have for general nilmanifold endomorphisms up till now is that $p(N^\Q_D) \subseteq X$. The following question is therefore natural.

\begin{QN}
\label{QN1}
Is there a description of $\Per(\bar{\alpha})$ for affine infra-nilmanifold endomorphisms $\bar{\alpha}: \Gamma \backslash G \to \Gamma \backslash G$? Equivalently, is there a description of the periodic points of nilmanifold endomorphisms which lie in the set $p(N^\Q)$?
\end{QN}
 
For infra-nilmanifold automorphisms, which have determinant $\pm 1$, we know that $\Per(\bar{\delta}) \supseteq p(N^\Q_1) = p(N^\Q)$ by Theorem \ref{perD}. This implies that for every affine infra-nilmanifold automorphism $\bar{\alpha}$ we have that $\Per(\bar{\alpha}) = \ePer(\bar{\alpha})$ by Theorem \ref{split}. This behavior completely describes these affine infra-nimanifold automorphisms.

\begin{Thm}
Let $\bar{\alpha}$ be an affine infra-nilmanifold endomorphism. The map $\bar{\alpha}$ is an affine infra-nilmanifold automorphism (so $\bar{\alpha}$ is a diffeomorphism) if and only if $\Per(\bar{\alpha}) = \ePer(\bar{\alpha}) \neq \emptyset$.
\end{Thm}

\begin{proof}
One direction was shown just above the theorem. For the other direction, it follows from Proposition \ref{topcon} and Theorem \ref{reduc} that it suffices to show this  in the case of a nilmanifold endomorphism $\bar{\delta}$. Assume that $\bar{\delta}$ is not an nilmanifold automorphism, so its determinant $D$ satisfies $\vert D \vert > 1$. Let $N \backslash G$ be the nilmanifold and consider the uniform lattice $N_0 = \delta^{-1}(N) \le G$ of $G$. Then $N$ is a finite index subgroup of $N_0$ of index $\vert D \vert$.

Consider the set $p(N_0)$ which is a set of order $\vert D \vert$. After applying the map $\bar{\delta}$ we get that $$\bar{\delta}\left(p\left(N_0\right)\right) = p\left(\delta\left(N_0\right)\right) = p(N)$$ and hence the only periodic point in $p(\tilde{N})$ is $N e$. Since $\vert D \vert > 1$, this implies that there exists an eventually periodic point which is not periodic.
\end{proof}

\section{Examples}
\label{periodexam}

In this section we compute the set of periodic points for some concrete examples of affine infra-nilmanifold endomorphisms. These examples demonstrate that in general it is a hard task to describe the set of periodic points, contrary to the eventually periodic points which are completely described in Theorem \ref{fulldes}.

For abelian groups, the relative order of elements in $N^\Q$ also helps to show that points are not periodic.

\begin{Prop}
\label{notper}
Let $\T^n = \Z^n\backslash \R^n$ be an $n$-torus with projection map $p: \R^n \to \Z^n$ and let $f_A$ be a toral endomorphism induced by the matrix $A \in \GL(n,\Q)$. Consider a periodic point $p(q)$ with $q \in \Q^n$, then $$\ordz_{\Z^n}\left(A^k(q)\right) = \ordz_{\Z^n}(q)$$ for all $k >0$.
\end{Prop}

\begin{proof}
Note that for abelian groups the set $$X_s = \left\{q \in \Q^n \suchthat \hspace{1mm} \ordz_{\Z^n}(q) \mid s\right\} \subseteq \Q^n$$ does form a subgroup of $\Q^n$. Because $\Z^n \le X_s$, we have that $p(q) \in p(X_s)$ if and only if $q \in X_s$. Assume that $s = \ordz_{\Z^n}\left(A^k(q)\right) < \ordz_{\Z^n}(q)$ for some $q \in \Q^n$. Then $p(q) \notin p(X_s)$ but $f_A^k\left(p\left(q\right)\right) \in p(X_s)$ and this last set is invariant under $f_A$. This implies that $p(q)$ is not periodic.
%
%
\end{proof}

For general nilpotent groups $N$, there is the following weaker version of this proposition.

\begin{Prop}
\label{computeper}
Let $\bar{\delta}$ be a nilmanifold endomorphism on the nilmanifold $N \backslash G$ induced by $\delta \in \Aut(G)$. Denote by $p: G \to N \backslash G$ the natural projection map and let $p(n)$ be a periodic point for $n \in N^\Q$. If $q$ is a prime such that $q \mid \ordn(n)$, then $q \mid \ordn(\delta^k(n))$ for all $k > 0$.
\end{Prop}
\begin{proof}
Assume that there exists a prime $q$ and a $k>0$ such that $q \mid \ordn(n)$, but $q \nmid \ordn(\delta^k(n))$. Write $s =\ordn(\delta^k(n))$, then we have that $n \notin N^{\frac{1}{s}}$ by Proposition \ref{index2}. Hence $p(n) \notin p(N^{\frac{1}{s}})$ since $N^{\frac{1}{s}}$ is a subgroup of $N^\Q$. The point $\bar{\delta}^k(p(n)) \in p(N^{\frac{1}{s}})$ and because $\bar{\delta}$ maps the set $p(N^{\frac{1}{s}})$ to itself, the point $p(n)$ cannot be periodic.
\end{proof}

Both Propositions \ref{notper} and \ref{computeper} are useful for showing that certain points are not periodic. We will apply these results in the following examples. In Theorem \ref{perD}, we showed that $p(N_D^\Q)$ is a subset of the periodic points for an infra-nilmanifold endomorphism of determinant $D$. In some cases, the set of periodic points will be equal to the subset $p(N_D^\Q)$.

\begin{Ex}
Let $\lien^\Q$ be a nilpotent Lie algebra which has a positive grading, so $$\lien^\Q = \bigoplus_{i > 0} \lien_i^\Q$$ with $[\lien_i^\Q,\lien_j^\Q] \subseteq \lien_{i+j}^\Q$ for all $i,j > 0$. For every prime $p$, there exists a Lie algebra automorphism $\varphi_p: \lien^\Q \to \lien^\Q$ given by $\varphi_p(X) = p^i X$ for all $X \in \lien^\Q_i$. Denote by $N^\Q$ the radicable nilpotent group corresponding to $\lien^\Q$. From \cite[Corollary 3.3]{dd14-1} it follows that there exists a full subgroup $N$ of $N^\Q$ such that $\varphi_p$ induces a group morphism $\psi_p: N \to N$. 

For every $n \in N^\Q$ with $p \mid \ordn(n)$, we have that $\ordn(\psi_p(n)) \mid \frac{\ordn(n)}{p}$. So for some power $\psi_p^k$ we have that $\gcd(\ordn(\psi_p^k(n)),p) = 1$. This implies that $\Per(\bar{\psi}_p) = p(N_p^\Q)$ and thus equality holds in Theorem \ref{perD}. In this case, $\ePer(\bar{\psi}_p) \setminus \Per(\bar{\psi}_p)$ is a dense subset of the nilmanifold.
\end{Ex}

Similarly there is a closed formula for the periodic points of toral endomorphisms induced by diagonal matrices on $\T^n$ by using Proposition \ref{notper} in the same way. In the case of nilmanifold endomorphisms which are not induced by a diagonal matrix, the situation gets a lot more complicated. 
\begin{Ex}
\label{exampl}
Consider the $2$-torus $\T^2 = \Z^2 \backslash \R^2$ with natural projection map $p: \R^2 \to \T^2$. Denote the nilmanifold endomorphism induced by a matrix $A \in \GL(2,\Q)$ as $f_A: \T^2 \to \T^2$. We will compute $\Per(f_A)$ for the matrices $$A_1 = \begin{pmatrix} 3 & 1 \\ 1 & 1 \end{pmatrix}, \hspace{1mm} A_2 = \begin{pmatrix} 4 & 1 \\ 2 & 1 \end{pmatrix},\hspace{1mm}  A_3 = \begin{pmatrix} 1 & 3 \\ -1 & 1 \end{pmatrix} \text{ and } \hspace{1mm}  A_4 = \begin{pmatrix} 5 & 2 \\ -1 & 1 \end{pmatrix}.$$ Since for every $k > 0$, $A^k_i - I$ is invertible over $\Q$, we know that $\ePer(f_{A_i}) = p(\Q^2)$ in all four cases. This implies that the periodic points also form a subset of $p(\Q^2)$. Note that the first two matrices are hyperbolic (so the maps $f_{A_i}$ are Anosov endomorphisms in these cases) whereas the last two matrices induce expanding maps. 
\begin{enumerate}
\item[\fbox{$A_1$}] For the first example, we have $$\left(A_1\right)^2 = \begin{pmatrix} 10 & 4 \\ 4 & 2 \end{pmatrix} = 2 \begin{pmatrix} 5 & 2 \\2 & 1 \end{pmatrix}$$ with $\det \left( \begin{pmatrix}  5 & 2 \\2 & 1 \end{pmatrix} \right) = 1$. Therefore, if $q \in \Q^2$ satisfies $2 \mid \ordt(q)$ then $\ordt\left(\left(A_1\right)^2\left(q\right)\right) = \frac{\ordt{q}}{2}$ and thus $p(q)$ is not periodic by Proposition \ref{notper}. By Theorem \ref{perD} we conclude that $$\Per(f_{A_1}) = p(N^\Q_2).$$
\item[\fbox{$A_2$}] In the second case, we claim that $\ePer(f_{A_2}) \setminus \Per(f_{A_2}) \supseteq p(X)$ with $X \subseteq \Q^2$ given by $$X = \left\{\left(\frac{q_1}{2^k},\frac{q_2}{2^l}\right) \mid q_1, q_2 \in \Q, \hspace{1mm} k > l \geqslant 0, \hspace{1mm}  \gcd(\ordz_{\Z}(q_i),2) = 1 \right\}.$$ Note that the set $X^c$ is invariant under $A_2$. To see that $p(X) \cap \Per(f_{A_2}) = \emptyset$, write $$A_2 = \begin{pmatrix} 4 & 1 \\ 2 & 1 \end{pmatrix} = \begin{pmatrix} 2 & 1 \\1 & 1 \end{pmatrix} \begin{pmatrix} 2 & 0 \\0 & 1\end{pmatrix}$$ where $\begin{pmatrix} 2 & 1 \\1 & 1 \end{pmatrix}$ has determinant $1$. If $x \in X$, then $\ordt\left(A_2(x) \right) = \frac{\ordt \left( x \right)}{2}$, so $p(x)$ is indeed not periodic by Proposition \ref{notper}. 

On the other hand, if $x = (q_1,q_2) \notin X$ with $2^k \mid \ordt(x)$, then $2^k \mid \ordz_{\Z}(q_2)$ and thus also $2^k \mid \ordz_{\Z}(2q_1 + q_2)$ which is the second component of $A_2(x)$. We conclude that if $x \notin X$, then $\ordt(x) = \ordt(A_2(x))$. So the points of relative order exactly $s$ in $X^c$ project to a (finite) set which is invariant under $f_{A_2}$. This implies by Lemma \ref{elem} that for every $s \in \N$, $f_{A_2}$ has a periodic point $p(x)$ with $\ordt(x) = s$. So $\Per(f_{A_2}) \supsetneq p(N^\Q_2)$.

\item[\fbox{$A_3$}] For the third matrix, we have that $$\left(A_3\right)^2 = \begin{pmatrix} -2 & 6 \\ -2 & -2 \end{pmatrix} = 2 \begin{pmatrix} -1 & 3 \\ -1& -1 \end{pmatrix}$$ and thus similarly as in the case $A_1$, we conclude that $\Per(f_{A_3}) = p(N^\Q_2)$. 
\item[\fbox{$A_4$}] In the last example, we find that $\ePer(f_{A_4}) \setminus \Per(f_{A_4}) \supseteq p(Y)$ with $$Y = \left\{ \left(\frac{q_1}{7^k},\frac{q_2}{7^k}\right) \suchthat k > 0, \hspace{1mm} \gcd(\ordz_{\Z}(q_i),7) = 1 = \gcd\left(\ordz_{\Z}\left(\frac{q_1 - q_2}{7}\right),7\right) \right\}.$$ Also for every $s \in \N$, we have a periodic points $p(x)$ such that $\ordt(x)=s$. The proof is similar as in the case $A_2$. \end{enumerate}
\end{Ex}
\noindent So these examples show that Theorem \ref{perD} sometimes but not always gives us the set of periodic points, even not in the special case of Anosov endomorphisms or expanding maps. Note that in all examples above, the set of eventually periodic points which are not periodic form a dense subset of the infra-nilmanifold. 

\begin{QN}
For which affine infra-nilmanifold endomorphism $\bar{\alpha}: \Gamma \backslash G \to \Gamma \backslash G$ does it hold that $\ePer(\bar{\alpha}) \setminus \Per(\bar{\alpha})$ is a dense subset of $\Gamma \backslash G$?
\end{QN}

\noindent The only known examples which do not satisfy this property are the affine infra-nilmanifold automorphisms.

\section{Generalization to other types of self-maps}
\label{moregeneralself}
In this section we discuss the generalization of some results in this paper to a more general class of self-maps on infra-nilmanifolds. 

Consider the semi-group $\aff(G) = G \rtimes \Endo(G)$ of affine maps on $G$ which are not necessarily invertible, where $\Endo(G)$ denotes the semi-group of all continuous group morphisms $G \to G$. This semi-group also acts in a natural way on $G$ by extending the natural action of $\Aff(G)$ on $G$, so every $\alpha = (g,\delta) \in \aff(G)$ acts as \begin{eqnarray*}
{}^\alpha h = g \delta(h) \hspace{2.5mm} \forall h \in G.
\end{eqnarray*}
Let $\alpha \in \aff(G)$ be an affine map such that $\alpha \Gamma \subseteq \Gamma \alpha$, or equivalently such that for every $\gamma \in \Gamma$, there exists some $\gamma^\prime \in \Gamma$ with $\alpha \gamma = \gamma^\prime \alpha$. The map $\alpha$ will induce a map $$\bar{\alpha}: \Gamma \backslash G \to \Gamma \backslash G$$ given by $\bar{\alpha}(\Gamma g) = \Gamma {}^\alpha g$. The following result shows that, up to homotopy, we get all possible maps on infra-nilmanifolds.

\begin{Thm}[\cite{lee95-2}]
Let $\varphi: \Gamma \to \Gamma$ be a group morphism of an almost-crystallographic group $\Gamma$. Then there exists $\alpha \in \aff(G)$ such that $$\varphi(\gamma) \alpha = \alpha \gamma$$ for all $\gamma \in \Gamma$.
\end{Thm}

These self-maps form the natural generalization of affine infra-nilmanifold endomorphisms and the techniques of this paper can also be applied to study their periodic points. Since the linear part of a map $\alpha \in \aff(G)$ is not always invertible, there are some difficulties in generalizing the previous results, in particular the results about periodic points where we use the injective version of Lemma \ref{elem}.

\subsection*{Reduction to nilmanifolds}

Let $\Gamma \backslash G$ is an infra-nilmanifold and $\bar{\alpha}: \Gamma \backslash G \to \Gamma \backslash G$ a map induced by $\alpha \in \aff(G)$. Contrary to the situation of affine infra-nilmanifold endomorphisms, $\bar{\alpha}$ does not always have a lift to the nilmanifold $N \backslash G$ with $N = \Gamma \cap G$. There does exist a nilmanifold though which finitely covers $\Gamma \backslash G$ and such that every map has a lift to this nilmanifold.

\begin{Thm}
Let $\Gamma \backslash G$ be an infra-nilmanifold, then there exists a nilmanifold $N_0 \backslash G$ such that every map $\bar{\alpha}: \Gamma \backslash G \to \Gamma \backslash G$ induced by an $\alpha \in \aff(G)$ has a lift $\tilde{\alpha}$ to $N_0 \backslash G$. Moreover, if $p: N_0 \backslash G \to \Gamma \backslash G$ is the covering map, then \begin{align*}
p^{-1}\left(\ePer(\bar{\alpha})\right)& =\ePer(\tilde{\alpha})\\
\Per(\bar{\alpha})& = p\left(\Per(\tilde{\alpha}\right).
\end{align*}
\end{Thm}

\noindent For giving a description of the set $\Per(\bar{\alpha})$ and studying its density, this weaker version of Theorem \ref{reduc} is enough. 

\begin{proof}
Let $F$ be the holonomy group of $\Gamma$ and take $$N_0 = \Gamma^{\vert F \vert} = \langle \gamma^{\vert F \vert} \mid \gamma \in \Gamma \rangle \le G.$$ Since $N_0$ is a fully characteristic subgroup, every map has a lift to $N_0 \backslash G$. The other claims follow immediately from Proposition \ref{cover}. 
\end{proof}

The other results of Section \ref{reduction2} generalize without problem. In particular, also Theorem \ref{topcon} is true for maps induced by affine maps of $\aff(G)$.
\begin{Thm}
Let $\bar{\alpha}$ be a map on the infra-nilmanifold $\Gamma \backslash G$, then $\bar{\alpha}$ is topologically conjugate to a map $\bar{\delta}$ induced by $\delta \in \Endo(G)$ on some infra-nilmanifold if and only if it has a fixed point.
\end{Thm}

\subsection*{Sufficient condition}

The results of this paragraph about eventually periodic points generalize immediately.

\begin{Thm}
\label{epergen}
Let $N$ be a lattice of the nilpotent Lie group $G$ with radicable hull $N^\Q$. If $\bar{\delta}: N \backslash G \to N \backslash G$ is a map induced by a group morphism $G \to G$ then $$p(N^\Q) \subseteq \ePer(\bar{\delta}).$$
\end{Thm}
\noindent The results about periodic points are harder to generalize since the linear part of $\alpha$ is not injective and therefore are not discussed here. It is certainly not true that the set of periodic points of such a map is dense in the manifold.

\subsection*{Necessary condition}

Since Theorem \ref{coset} is true for general group morphisms of $\mathcal{F}$-groups, we can still use this result in the generalized case.

\begin{Thm}
Let $\Gamma \backslash G$ be an infra-nilmanifold and $\bar{\alpha}: \Gamma \backslash G \to \Gamma \backslash G$ a map induced by the affine map $\alpha = (g,\delta) \in \aff(G)$. If $\bar{\alpha}$ has a periodic point $\Gamma g_0$, then $$\ePer(\bar{\alpha}) = p(N^\Q g_0 H^\R)$$ with $$H^\R = \{h \in G \suchthat \exists k_1, k_2 > 0:  \delta^{k_1} (h) =\delta^{k_2}(h)\}.$$
\end{Thm}
\noindent The subgroup $H^\R$ corresponds to the Lie subalgebra given by the generalized eigenspace of the eigenvalue $0$ and the eigenspaces of eigenvalues which are roots of unity. 

\begin{proof}
The proof is identical to the proof of Theorem \ref{fulldes}, except for the step where we use that $\delta^{k_1}: N^\Q \to N^\Q$ is invertible. Instead, we use Theorem \ref{coset} directly on the maps $\delta^{k_1}$ and $\delta^{k_2}$ to get the result.
\end{proof}

So the situation for eventually periodic points can be completely generalized to maps induced by an affine map $\alpha \in \aff(G)$. The following question is still completely open.

\begin{QN}
How does the set of (eventually) periodic points look like for a map induced by an affine map $\alpha \in \aff(G)$ on an infra-nilmanifold modeled on $G$?
\end{QN}

Note that this question is a generalization of Question \ref{QN1} to the bigger class of self-maps introduced in this section.
\bibliography{ref}
\bibliographystyle{plain}

\end{document}